\documentclass[12pt,oneside,reqno]{amsart}
\usepackage{txfonts}
\usepackage{bbm}
\usepackage{cases}
\usepackage{amsmath}
\usepackage{graphicx}
\usepackage{mathrsfs}
\usepackage{stmaryrd}
\usepackage{amsfonts}
\usepackage{xcolor}
\usepackage{enumerate,amsmath,amssymb,amsthm}

\oddsidemargin=0mm \topmargin=-12mm \numberwithin{equation}{section}

\numberwithin{equation}{section}
\newtheorem{thm}{Theorem}[section]

\newtheorem{proposition}{Proposition}[section]
\newtheorem{corollary}{Corollary}[section]
\newtheorem{lemma}{Lemma}[section]
\newtheorem{definition}{Definition}[section]
\newtheorem{Rem}{Remark}[section]

\newtheorem{assumption}{Assumption}[section]

\def\eps{\varepsilon}

\def\ti{\tilde}

\def\z{\zeta}

\def\<{{\langle}}
\def\>{{\rangle}}
\def\({{\Big(}}
\def\){{\Big)}}
\def\]{{\Big]}}
\def\[{{\Big[}}

\def\bx{{\mathbf{x}}}

\def\dif{{\mathord{{\rm d}}}}

\def\min{{\mathord{{\rm min}}}}

\def\={&\!\!=\!\!&}
\def\bt{\begin{theorem}}
\def\et{\end{theorem}}
\def\bl{\begin{lemma}}
\def\el{\end{lemma}}
\def\br{\begin{remark}}
\def\er{\end{remark}}

\def\bd{\begin{definition}}
\def\ed{\end{definition}}
\def\bp{\begin{proposition}}
\def\ep{\end{proposition}}
\def\bc{\begin{corollary}}
\def\ec{\end{corollary}}
\def\bx{\begin{Examples}}
\def\ex{\end{Examples}}

\def\cB{{\mathcal B}}

\def\cD{{\mathcal D}}

\def\cK{{\mathcal K}}
\def\cL{{\mathcal L}}

\def\mE{{\mathbb E}}

\def\mN{{\mathbb N}}

\def\mP{{\mathbb P}}

\def\mZ{{\mathbb Z}}

\def\geq{\geqslant}
\def\leq{\leqslant}

\numberwithin{equation}{section}

\title{\bf{Ergodicity of the 2D Navier-Stokes Equations with Degenerate  Multiplicative  Noise$^*$}\thanks{The authors
were Supported by 973 Program, No. 2011CB808000 and
Key Laboratory of Random Complex Structures and Data Science, No.2008DP173182, NSFC, No.:10721101, 11271356, 11371041.}}

{\author{   Xuhui Peng  \\
{\em\tiny   College of Mathematics and Computer Science, Hunan Normal University,
 }\\
{\em\tiny  Changsha, {\rm 410081}, P.R.China.}\\
}}

\date{}
\begin{document}
\footnote{ Email addresses:  pengxuhui@amss.ac.cn(Xuhui Peng) }
\footnote{$^*$ The authors
were Supported by 973 Program, No. 2011CB808000 and
Key Laboratory of Random Complex Structures and Data Science, No.2008DP173182, NSFC, No.:10721101, 11271356, 11371041.}

\maketitle

\begin{abstract}
Consider the two-dimensional, incompressible
Navier-Stokes equations on the torus $T^2=[-\pi,\pi]^2$ driven by a
degenerate multiplicative noise
\begin{equation}
\label{1-3}
dw_{t}=  \nu \Delta w_{t}dt+B(\mathcal{K} w_{t},w_{t})dt+Q(w_t)dB_t.
\end{equation}
We use the Malliavin calculus to  prove that the semigroup $\{P_t\}_{t\geq 0}$ generated by the
solutions to $(\ref{1-3})$ is  asymptotically strong Feller.
Moreover, we  use the coupling method to   prove  the  semigroup $\{P_t\}_{t\geq 0}$ is
exponentially ergodic in some sense. Our result is stronger than that in
\cite{O2}.

 \vspace{1mm}
\end{abstract}

\section{Introduction, Preliminaries and  Properties for  Solution}
\label{}

\subsection{Introduction and Main Results}

This work is motivated by  paper \cite{martin}, in which, Martin
Hairer and Jonathan C. Mattingly considered  the following
two-dimensional, incompressible Navier-Stokes equations on the torus
$T^2=[-\pi,\pi]^2$ driven by  a additive degenerate noise
\begin{equation}
\label{1-1}
\dif  w_{t}=\nu \Delta w_{t} \dif t+B(\mathcal{K} w_{t},w_{t})dt+Q\dif  B(t).
\end{equation}
With the  asymptotically strong Feller property that they
discovered, they proved the uniqueness and existence of the
invariant measure for the semigroup generated by the solution to
$(\ref{1-1})$.  Also, in \cite{Hairer02} they proved the solution has a spectral gap in Wasserstein distance.  In this article,  we consider the same questions for  the   following
two-dimensional, incompressible stochastic  Navier-Stokes equations with degenerate  multiplicative  noise
\begin{eqnarray}
\label{1-2} \dif  w_{t}=\nu \Delta w_{t}\dif t+B(\mathcal{K} w_{t},w_{t})\dif t+Q(w_t)\dif  B(t).
\end{eqnarray}

We assume $B_t$ is a cylindrical Wiener process of a Hilbert space $U$, that is  there  exist
 an orthonormal basis  $(\beta_n)$ of $U$, and  a family $(B_n)$ of independent brownian
motions such that
\begin{eqnarray*}
  B_t=\sum_{n=1}^{\infty}B_n(t)\beta_n.
\end{eqnarray*}
Denote $H = L^2_0$, the space of real-valued square-integrable functions on the torus with vanishing mean and $\|\cdot\|$ denotes the $L^2$-norm on the Hilbert space $H.$
 We make  the following Hypotheses  in this article on $Q$:

\begin{description}
  \item[\textbf{H1}]  The function $Q:H\rightarrow \cL_2(U;H)$  is bounded.
  Assume
 \begin{align*}
   B_0=\sup_{u\in H}\|Q(u)\|^2_{\cL_2(U;H)}<\infty,
 \end{align*}
   \item[\textbf{H2}] The function $Q:H\rightarrow \cL_2(U;H)$  is   Lipschitz and denote
  $L_Q$ be  the Lipschitz constant of Q.

  \item[\textbf{H3}]   There exist $N\in \mN^*$ and a bounded measurable map $g : H\rightarrow L(H;U)$ such that
for any $u\in H$
$$
Q(u)g(u) = P_N,
$$
\end{description}
here the meaning of $P_N$ can see subsection 1.2. In the following,
we set
 \begin{align*}
   \|Q(u)\|=\|Q(u)\|_{\cL_2(U;H)}.
 \end{align*}
Denote by
$a\vee b:=\max\{a,b\},
a\wedge  b:=\min\{a,b\}.$ Let $C(d)$ be some positive constant depending on $d,B_0,L_Q$ and $\nu,$ and let  $C$ be some positive constant depending on $B_0,L_Q$ and $\nu.$ The constant $C$ or $C(d)$ may changes from line to line.

Define $V_\eta(x)=e^{\eta \|x\|^2}, \forall \eta>0$, and for any  $r\in (0,1]$ and $\eta>0,$ we introduce a family of distances $\rho_r^{\eta}$ on H
\begin{eqnarray*}
  \rho_r^\eta(x,y):=\inf_{\gamma}\int_0^1 V_\eta^r(\gamma(t))\|\dot{\gamma}(t) \|\dif t,
\end{eqnarray*}
where the infimum runs over all paths $\gamma$ such that $\gamma(0)=x$ and $\gamma(1)=y.$  For simplify of writing, we denote $ \rho_r(x,y):=\rho_r^\eta(x,y).$
Let  $P_t$  be  the transition semigroups of
${w_t}$, that is
\begin{align*}
P_tf(w_0):=\mathbb{E}_{w_0}f(w_t),~~f\in\mathscr{B}_b(H).
\end{align*}
Our main results in this article are  the following three  Theorems.
\begin{thm}[Weak Form of Irreducibility]\label{2015-7}
 There exists some  $N_0$ such that
if    \textbf{H1,H2} hold and  \textbf{H3} holds  $N\geq N_0$, then the solutions to   (\ref{1-2}) have the weak form of irreducibility. That is given any $\eta\in (0,\frac{\nu}{8B_0})$, $C>0,$ $r\in (0,1)$ and $\delta>0$, there exists a $T_0$ such that for any $T\geq T_0$ there exists  an $a>0$ such that
\begin{eqnarray*}
  \inf_{\|x\|,\|y\|\leq C}\sup_{\pi\in \Gamma(P_T^*\delta_{x}, P_T^*\delta_{y})}\pi\left\{(x',y')\in H\times H, \rho_r(x',y') \leq \delta\right\}>a.
\end{eqnarray*}
\end{thm}

\begin{thm}[Gradient Estimate]
\label{1-17}
There exists $\eta_0>0,$ such that for any $\eta\leq \eta_0$,
there exist some constant  $N_0:=N_0(\eta)>0$,  such that if
Hypotheses  \textbf{H1},\textbf{H2} hold and    \textbf{H3} holds with   $N\geq N_0$, then
for any $f$
  \begin{eqnarray*}
 |\nabla P_t f (w_0)| &\leq& C(N) \exp{\(\big(\frac{4\eta}{\nu} +\eta \big)\|w_0\|^2\)}\sqrt{P_t|\varphi|^2(w_0)}+
Ce^{\frac{4\eta}{\nu}\|w_0\|^2}e^{-\nu N^2 t} \sqrt{P_t\|D\varphi\|^2(x)}.
  \end{eqnarray*}
\end{thm}

We in introduce the following family of norms
\begin{eqnarray*}
  \|\phi\|_{V_\eta^r}=\sup_{x\in H}\frac{|\phi(x)|+\|D\phi(x)\|}{V_\eta^r(x)}
\end{eqnarray*}
When we take $r=1,$  we simply write $\|\phi\|_{V_\eta}$.
\begin{thm}[Exponential Mixing]\label{4-1}
For any $\eta\in (0,\frac{\nu}{16B_0})$,
there exists some  $N_0:=N_0(\eta)$ such that
if    \textbf{H1,H2} hold and  \textbf{H3} holds  $N\geq N_0$,
then there exists a unique invariant probability  measure $\mu^*$ for $P_t$ and   positive  constants $\theta$ and $C$  such that for any  $\phi\in \cB,$
\begin{eqnarray}\label{1-16}
   \|P_t\phi-\mu^* \phi\|_{V_\eta} &\leq &  Ce^{-\theta t} \|\phi-\mu^*\phi\|_{V_\eta}, \ \ \forall t>0.
\end{eqnarray}
\end{thm}

\begin{Rem}
 Odasso \cite[Theorem 3.4]{O2} used  coupling method to  establish exponential mixing of the solutions of
stochastic Navier-Stokes equation under hypotheses  \textbf{H1,H2} and  \textbf{H3}, but  he didn't give the proof of  asymptotically strong Feller property and the  weak form of irreducibility.
The results of   theorem \ref{4-1}  is stronger  than  that in  \cite{O2}.
\end{Rem}

\subsection{Preliminaries}\
\label{2014-2}
Recall that the Navier-Stokes  equations are given by
 \begin{eqnarray*}
 \partial_t u+(u\cdot \nabla) u=\nu \triangle  u-\nabla p+\xi,\ \ \text{div}~u=0.
 \end{eqnarray*}
where $\xi(x,t) $ is the external force field acting  on the fluid.

 The vorticity  $w$ is defined by $w=\nabla \wedge u=\partial_2 u_1-\partial_1 u_2.$
$B(u,w)=-(u\cdot \nabla )w .$ For $k=(k_1,k_2)\in \mathbb{Z}^2 \setminus  \{(0,0)\},\ k^{\perp
}=(k_2,-k_1), ~w_k
=\langle w ,(2 \pi)^{-1}e^{ik\cdot x}\rangle_{H}.$ The
operator $\mathcal{K}$ is defined in Fourier space by
$$
(\mathcal{K}w)_k=\langle \mathcal{K}w, (2
\pi)^{-1}e^{ik\cdot x } \rangle_H=-iw_k k^{\perp }/||k||^2.
$$
We write $\mathbb{Z}^2\setminus \{(0,0)\}=\mathbb{Z}^2_{+}\cup \mathbb{Z}^2_{{-}}$, where
\begin{eqnarray*}
&&\mathbb{Z}_{+}^2=\left\{(k_1,k_2)\in \mathbb{Z}^2:\  k_2>0\} \cup \{ (k_1,0)\in \mathbb{Z}^2:\
k_1>0\right\},
\\&&\mathbb{Z}_{-}^2=\left\{(k_1,k_2)\in \mathbb{Z}^2:\  -(k_1,k_2) \in \mathbb{Z}_{+}^2 \right\},
\end{eqnarray*}
and set,  for $k\in  \mathbb{Z}^2 {  \setminus } \{(0,0)\}$,
\begin{eqnarray*}
e_k(x)=\left\{ \begin{matrix}\text{sin}(k\cdot x) & k \in \mathbb{Z}_{+}^2,\\
\text{cos} (k \cdot x)&  k \in \mathbb{Z}_{-}^2.\end{matrix} \right.
\end{eqnarray*}
Then,  $\{e_k, k\in \mZ^2\setminus\{(0,0)\}$ an  orthonormal basis  of $H$.
For $\alpha \in \mathbb{R} $ and   a smooth function $w$ on $[-\pi,\pi]^2$
with mean 0, denote $\|w\|_{\alpha}$ by
\begin{eqnarray*}
\|w\|_{\alpha}^2=\sum_{k \in \mathbb{Z}^2 \backslash
\{(0,0)\}}|k|^{2\alpha}|w_k|^2,\ \ w_k=\langle w,(2
\pi)^{-1}e^{ik\cdot x}\rangle_{H},
\end{eqnarray*}
and $\|w\|:=\|w\|_0$.   Denote  $H^{2\alpha}$ be the closure of smooth function   with respect to the norm $\|\cdot \|_{\alpha}$ in H.  We denote  $P_N$ and $Q_N$ the orthogonal projection in $H$ onto the space Span$\{e_k, 1\leq |k|\leq N\}$ and onto its complementary.  Denote $H_N=P_NH$.

Actually,    $(\ref{1-2})$  defines a stochastic flow on $H$. That means a family of continuous map
$\Phi_t:W \times H\rightarrow H$ such that $w_t=\Phi_t(B,w_0)$ is the solution
to $(\ref{1-2})$  with initial condition $w_0$ and noise $B.$

Given a  $v \in
L^2_{loc}(\mathbb{R}^+,U)$, the Malliavin derivative of
the $H$-valued random variable $w_t$ in the direction $v$, denoted by
$\mathcal{D}^vw_t$ is defined by
\begin{eqnarray*}
\mathcal{D}^{v} w_t=\lim_{\varepsilon \rightarrow 0}
\frac{\Phi_t(B+\varepsilon V,w_0)- \Phi_t(B,w_0) }{\varepsilon},
\end{eqnarray*}
where the limit holds almost surely with respect to Wiener measure
and $V(t)=\int_0^t v(s)ds$.
Let $\{J_{s,t}\}_{s\leq t}$  be the derivative flow between times s
and t, i.e for every $\xi \in H,J_{s,t}\xi$ is the solution of
\begin{eqnarray}
\label{0-1} \begin{cases}dJ_{s,t}\xi =\nu  \triangle J_{s,t}\xi dt +\tilde{B}(w_t,J_{s,t}\xi )dt
+DQ(w_t)J_{s,t}\xi dB_t,\\
J_{s,s}\xi=\xi,
\end{cases}
\end{eqnarray}
where $ \tilde{B}(w, u)=B(\mathcal{K}w,u)+B(\mathcal{K}u,w).$
 $J_{0,t}\xi$ is the effect on
$w_t$ of an infinitesimal perturbation of the initial condition in
the direction $\xi$. $DQ$ is Fr\'{e}chet derivation of $Q.$
Observe that $\mathcal{D}^{v} w_t=A_{0,t}v$, where
$A_{s,t}:L^2([s,t],U)\rightarrow H$
\begin{eqnarray}
\label{0-2}
A_{s,t}v=\int_s^tJ_{r,t}Q(w_r)v(r)dr.
\end{eqnarray}

If $B(u,v)=(u \cdot \nabla) v,\mathcal{S}=\left\{ (s_1,s_2,s_3)
\in\mathbb{R}_{+}^3:\sum s_i \geq 1,s \neq(1,0,0),(0,1,0),(0,0,1)
\right\} $.   Then the following relations are useful. Its  proof
can be seen  in \cite{P}  or \cite{martin} .
\begin{eqnarray}
\label{1-5} &\langle B(u,v),w\rangle=-\langle B(u,w),v\rangle,   \ \
&\text{if } \ \nabla \cdot  u=0,
\\ \label{1-6}& |\langle B(u,v),w\rangle| \leq C\|u\|_{s_1}\|v\|_{1+s_2}\|w\|_{s_3},&(s_1,s_2,s_3) \in \mathcal{S},
\\ \label{1-7} &\| \mathcal{K}w\|_{\alpha}=\|w\|_{\alpha-1},&
\\ \label{1-8} &\|w\|_{\frac{1}{2}}^2 \leq \|w\|_1 \|w\|.&
\end{eqnarray}
\subsection{Properties  For Solution}
In this subsection, we will give some Lemmas and Propositions  which will be used  in section 2 and section 3.

\begin{lemma}(\cite[Lemma A.1]{mat02})
\label{2-1} Let $M(s)$~be a continuous martingale with quadratic
variation~$[M,M](s)$~such that~$\mathbb{E}[M,M]< \infty$. Define the
semi-martingale  $N(s)=-\frac{\alpha}{2}[M,M](s)+M(s)$ for any
$\alpha>0$. If  $\gamma \geq 0,$ then for any $\beta > 0$ and $T
>\frac{1}{\beta}$
\begin{eqnarray*}
\mathbb{P}\left\{ \sup_{t \in [T-\frac{1}{\beta},T]}\int_0^t e^{-\gamma(t-s)} dN(s)   >\frac{e^{\frac{\gamma}{\beta}}}{\alpha} K  \right\}<e^{-K}.
\end{eqnarray*}
Specially,
\begin{eqnarray*}
\mathbb{P}\left\{\sup_{t}N(t) > \frac{1}{\alpha}K\right\}<e^{-K}.
\end{eqnarray*}
\end{lemma}

\begin{lemma}
\label{16} Assume \textbf{H1}. For every  $\eta \leq \frac{\nu}{2B_0}$, there
exists  a  constant $C=C(\nu,B_0)$ such that
\begin{equation*}
\mathbb{E}\exp{\{\sup_{t\geq 0}(\eta   \Arrowvert  w_{t} \Arrowvert^2)\}} \leq C e^{\eta \Arrowvert w_{0}\Arrowvert^2},
\end{equation*}
and
\begin{eqnarray*}
  \mE\[e^{\eta \|w_t\|^2}\]\leq Ce^{\eta e^{-\nu t}\|w_0\|^2}.
\end{eqnarray*}

\end{lemma}
\begin{proof}
By  It\^{o}'s  formula,
\begin{equation}
\label{2-2}
d  \eta \Arrowvert  w_{t} \Arrowvert^2+2\eta \nu \| w_{t} \Arrowvert^2_1dt
=2\eta \langle w_t,Q(w_t)dB_t\rangle +\eta \| Q(w_t) \|^2dt.
\end{equation}
Using the fact that $\|w_t\Arrowvert\leq \|w_t
\Arrowvert_1 $ and  $\|Q(w_t)\|^2\leq B_0,$
\begin{equation*}
d\eta \| w_{t} \Arrowvert^2+\nu \eta \| w_{t} \Arrowvert^2dt\leq 2\eta  \langle w_t,,Q(w_t)dB_t\rangle+\eta B_0 dt-\nu \eta \| w_{t} \Arrowvert^2dt,
\end{equation*}
that is,
\begin{equation*}
\eta d(\| w_{t} \Arrowvert^2 e^{\nu t})\leq 2\eta e^{\nu t} \langle w_t,Q(w_t)dB_t\rangle+  \eta B_0  e^{\nu t}dt-\nu \eta e^{\nu t} \| w_{t} \Arrowvert^2 dt.
\end{equation*}
So,
\begin{equation*}
\eta \| w_{t} \Arrowvert^2- \eta e^{-\nu t} \| w_0 \Arrowvert^2
 - \frac{\eta B_0 }{\nu}\leq 2\eta \int_0^t  e^{-\nu(t-s)}\langle w_s,Q(w_s)dB_s\rangle-\eta \nu \int_0^t  e^{-\nu(t-s)}  \| w_{s} \Arrowvert^2 ds.
\end{equation*}
By   Lemma $\ref{2-1},$   when  $\eta \leq  \frac{\nu}{2B_0},$
one arrives  at
\begin{equation}\label{2-12}
\mathbb{E}\exp{\left\{\sup_{t\geq 0}\left(\eta   \| w_{t} \Arrowvert^2 -\eta e^{-\nu t}  \| w_0 \Arrowvert^2- \frac{\eta B_0}{\nu}\right)\right\}}\leq 2,
\end{equation}
here we use the fact that if a random variable $X$ satisfies  $\mathbb{P}(X\geq C)\leq \frac{1}{C^2}$ for all $C\geq 0$, then $\mathbb{E}X\leq 2.$
Then  this lemma   follows by $(\ref{2-12})$.
\end{proof}

\begin{lemma}
\label{10} Assume  \textbf{H1}.
 For every  $\eta \leq \frac{\nu}{2B_0}$, there
exists  a  abosolute constant $C$ such that
\begin{equation*}
\mathbb{E}\exp{\(   \eta \sup_{t\geq 0}\big( \| w_{t} \Arrowvert^2
+\nu \int_0^t \| w_r \Arrowvert^2_1dr-B_0 t\big)    \)} \leq  C  \exp{(\eta \| w_0 \Arrowvert^2)}.
\end{equation*}
\end{lemma}
\begin{proof}
Combining   $(\ref{2-2})$ with   $\|w_t\Arrowvert\leq
\|w_t \Arrowvert_1 $,
\begin{equation*}
\eta \| w_t \Arrowvert^2 +\eta \nu \int_0^t \| w_r \Arrowvert^2_1dr-\eta \int_0^t\| Q(w_r) \Arrowvert^2 dr
-\eta \| w_0 \Arrowvert^2\leq 2 \eta \int_0^t\langle w_r,Q(w_r)dB_r\rangle-\eta \nu \int_0^t \| w_r \Arrowvert^2dr.
\end{equation*}
Therefore
\begin{equation*}
\eta \| w_t \Arrowvert^2 +\eta \nu \int_0^t \| w_r \Arrowvert^2_1dr
-\eta B_0 t-\eta \| w_0 \Arrowvert^2\leq 2 \eta \int_0^t\langle w_r,Q(w_r)dB_r\rangle-\eta \nu \int_0^t \| w_r \Arrowvert^2dr.
\end{equation*}
Due to  Lemma  $\ref{2-1}$, when    $\eta \leq \frac{\nu}{2B_0}$, for some absolutely constant $C,$
\begin{align*}
\mathbb{E}\exp{ \Big(   \eta \sup_{t\geq 0} \big(  \| w_{t} \Arrowvert^2
+\nu \int_0^t \| w_r \Arrowvert^2_1 dr-B_0 t
- \| w_{0} \Arrowvert^2  \big)    \Big) } \leq C,
\end{align*}
from which this  lemma follows.
\end{proof}

\section{Proof of Weak Form of Irreducibility}
Let
\begin{eqnarray*}
  F(w)=\nu \Delta w+B(\cK w, w).
\end{eqnarray*}
and $w=w(t,B,w_0^1)$ be the solution to  following equation
\begin{eqnarray}
\label{p-1}
\left\{
\begin{split}
 & \dif  w_{t}=\nu \Delta w_{t}\dif t+B(\mathcal{K} w_{t},w_{t})\dif t+Q(w_t)\dif  B(t),
 \\ &  w(0)=w_0^1.
\end{split}
\right.
\end{eqnarray}
Let $\tilde{w}=\tilde{w}(t,B,w_0^2)$ be the solution to the following equation,
\begin{eqnarray}
\label{p-2}
\left\{
\begin{split}
  & \dif \tilde{w}=F(\tilde{w})\dif t+KP_N(\tilde{w}-w(t,B,w_0^1))\dif t+ Q(\tilde{w}_t)\dif B_t,
  \\  & \ti{w}(0)=w_0^2.
  \end{split}
  \right.
\end{eqnarray}
Therefore
\begin{eqnarray*}
  \tilde{w}(t,B,w_0^2)=w(t,B+\int_0^\cdot h_s\dif s,w_0^2),
\end{eqnarray*}
here $h:H\times H\rightarrow U$ is given by
\begin{eqnarray*}
 h_s:=h(s,B):=-K g(\tilde{w}_s)P_N(\ti{w}_s-w(s,B,w_0^1)).
\end{eqnarray*}
Denote
\begin{eqnarray*}
\rho_r'(x,y)=\int_0^1 e^{r\eta \|ty+(1-t)x\|^2}\|x-y\|\dif t.
\end{eqnarray*}

\begin{lemma}
  There exists $C_1$  and $\kappa>0,$ such that
  \begin{eqnarray*}
    \mE(|w(t,B,w_0)|^2)\leq e^{-\kappa t}|w_0|^2+C_1
  \end{eqnarray*}
\end{lemma}

\begin{lemma}\label{2015-5}
There exists $C>0$ such that for any $w_0^1,w_0^2$ in H satisfying
\begin{eqnarray}\label{p-4}
  \|w_0^1\|^2+\|w_0^2\|^2\leq 2C_1
\end{eqnarray}
and  there exists $\gamma_1,\gamma_2>0,$  such that  for any $t\geq 0$, we have
\begin{eqnarray*}
  && \mP\(\rho_r'(w(t,B_2,w_0^2), w(t,B_1,w_0^1))  \geq Ce^{-\gamma_1 t},
  \\ &&  \ \ \ \ \ \  \tilde{w}(\cdot,B_1,w_0^2)=w(\cdot,B_2,w_0^2) \text{ on } [0,t]\)\leq Ce^{-\gamma_2 t}.
\end{eqnarray*}
\end{lemma}
\begin{proof}
Denote $x=\tilde{w}(t,B_1,w_0^2)$ and $y=w(t,B_1,w_0^1))$
\begin{eqnarray*}
&& \mP\( \rho_r'(w(t,B_2,w_0^2), w(t,B_1,w_0^1))  \geq Ce^{-\gamma_1 t},
  \\ &&  \ \ \ \ \ \  \tilde{w}(\cdot,B_1,w_0^2)=w(\cdot,B_2,w_0^2) \text{ on } [0,t]  \)
  \\ &&\leq  \mP\( \rho_r'(\tilde{w}(t,B_1,w_0^2), w(t,B_1,w_0^1))  \geq Ce^{-\gamma_1 t}\)
  \\ &&\leq \mP\(e^{2r\eta \|x\|^2+{2r\eta \|y\|^2 }}\|x-y\|\geq Ce^{-\gamma_1 t} \)
  \\ &&\leq \frac{1}{C}e^{\gamma_1 t} \mE\[\|x-y\|^2\]^{1/2}\mE\[ e^{4r\eta \|x\|^2+{4r\eta \|y\|^2}}\]^{1/2},
\end{eqnarray*}
and therefore this Lemma follows by Lemma \ref{10} and Lemma \ref{p-3}.
\end{proof}

\begin{lemma}\label{2015-3}
  There exists $p_1>0$ such that for any $w_0^1,w_0^2\in H$ satisfying (\ref{p-4}), we have \
  \begin{eqnarray*}
    \mP\(\int_0^{\infty}|h(t,B)|^2\dif t\leq C \)\geq p_1.
  \end{eqnarray*}

 \end{lemma}

Define
\begin{eqnarray*}
  \tau(B)=\inf\left\{t>0, \ \ \int_0^t|h(t,B)|^2\dif t>2C\right\}.
\end{eqnarray*}
Apply \cite[corollary1.5]{O2} to
\begin{eqnarray*}
  \left(B,B, B+\int_0^{\tau(B)\wedge \cdot }h(t,B)\dif t\right)
\end{eqnarray*}
we obtain $B_1,B_2$ cylindrical Wiener processes such that
\begin{eqnarray*}
  \left(B_2, B_1+\int_0^{\tau(B)\wedge \cdot }h(t,B_1)\dif t  \right)
\end{eqnarray*}
is a maximal coupling of $ \left(\cD(B), \cD(B+\int_0^{\tau(B)\wedge \cdot }h(t,B)\dif t)  \right)$
on $[0,\infty)$

\begin{lemma}\label{2015-4}
\begin{eqnarray*}
  \mP\(\ti{w}(\cdot ,B_1,w_0^2)=w(\cdot,B_2,w_0^2)\)\geq \frac{p_1}{4e^{2C}},
\end{eqnarray*}
\end{lemma}
\begin{proof}

  Let us set
  \begin{eqnarray*}
  A&=& \left\{B:   \tau(B)=\infty \right\}
  \\
    \Lambda_1&=&\cD(B),
    \\
    \Lambda_2&=&\cD(B+\int_0^{\tau(B)\wedge \cdot }h(t,B)\dif t),
  \end{eqnarray*}
  Novikov condition is obviously verifiede. So, Girsanov Transform gives
  \begin{eqnarray*}
    \(\frac{\dif \Lambda_1}{\dif \Lambda_2}\)(B)=\exp\(-\int_0^{\tau(B)}h(t,B)\dif B_t-\frac{1}{2}\int_0^{\tau(B)}|h(t,B)|^2\dif t\),
  \end{eqnarray*}
  which yields
  \begin{eqnarray*}
    \int_A  \(\frac{\dif \Lambda_1}{\dif \Lambda_2}\)^2 \dif \Lambda_1\leq \mE\[e^{\int_0^{\tau(B)} |h(t,B) |^2\dif t}\]\leq e^{2C}.
  \end{eqnarray*}
By Lemma \ref{2015-3}
  \begin{eqnarray*}
    \Lambda_1(A)\geq p_1.
  \end{eqnarray*}
  By \cite[Lemma 1.3]{O2}
  \begin{eqnarray*}
    \Lambda_1\wedge \Lambda_2(A)\geq \frac{p_1}{4e^{2C}},
  \end{eqnarray*}
 combine it  with \cite[Lemma 1.2]{O2} yields the result of this Lemma.
\end{proof}

Now we are in the position of the proof of Theorem  \ref{2015-7}.
\begin{proof}
\begin{eqnarray*}
&& \sup_{\pi\in \Gamma(P_T^*\delta_x, P_T^*\delta_y)}\pi\left\{(x,y)\in H\times H, \rho_r(x,y) \leq \delta\right\}
\\ && \geq \mP\(\rho_r'(w(t,B_2,w_0^2),w(t,B_1,w_0^1)) \leq \delta \)
\\
  && \geq \mP\(\rho_r'(w(t,B_2,w_0^2),w(t,B_1,w_0^1)) \leq Ce^{-\gamma_1 t}\)
  \\ && \geq  \mP\(\ti{w}(\cdot ,B_1,w_0^2)=w(\cdot,B_2,w_0^2)\)
  \\ && \ \ \ \  -\mP\(\rho_r'(w(t,B_2,w_0^2)-w(t,B_1,w_0^1))\geq Ce^{-\gamma_1 t},
  \\ &&  \ \ \ \ \ \  \ \ \ \ \   \tilde{w}(\cdot,B_1,w_0^2)=w(\cdot,B_2,w_0^2) \text{ on } [0,t]\)
\end{eqnarray*}
Combing it with Lemma \ref{2015-5} and Lemma \ref{2015-4} yields
\begin{eqnarray*}
  \sup_{\pi\in \Gamma(P_T^*\delta_x, P_T^*\delta_y)}\pi\left\{(x',y')\in H\times H, \rho_r(x',y') \leq \delta\right\} \geq \frac{p_1}{4e^{2C}}- Ce^{-\gamma_2 t},
\end{eqnarray*}
which finishes the proof of Theorem \ref{2015-7}.
\end{proof}

\section{Proof of Gradient Estimate}
For any  $v\in  L^2_{loc}(\mathbb{R}^+,U)$ and $\xi\in H$ with  $\|\xi\|=1$,   denote  $\rho_t=J_{0,t}\xi -\mathcal{D}^{v}w_t=J_{0,t}\xi -A_{0,t}v$.
Then   $\rho_t$ satisfies the following equation
\begin{equation}\label{3-32}
\dif \rho_t=\nu \Delta  \rho_t\dif t+\tilde{B}(w_t, \rho_t)\dif t+DQ(w_t) \rho_t \dif B_t-Q(w_t)v_t\dif t.
\end{equation}
Let  $\zeta_t $ be the  solution to the following equation
\begin{eqnarray}\label{1-18}
\left\{
 \begin{split}
d\zeta_{t}&=DQ(w_t)\zeta_t\dif B_t+B_1\Delta \zeta_t^{l}\dif t
\\ & \ \ \ +\pi_h \tilde{B}(w_t,\zeta_t)\dif t+\nu \Delta \zeta_t^h\dif t,
\\ \zeta_{0}&=\xi,
 \end{split}\right.
 \end{eqnarray}
 here  $B_1$ is a  constant bigger than $\nu N^2$ and
 $\zeta_t^{l}=P_{N}\zeta_t,\  \zeta_t^{h}=Q_{N}\zeta_t, \pi_h \tilde{B}(w_t,\zeta_t)=Q_N \tilde{B}(w_t,\zeta_t).$
We set the infinitesimal perturbation $v$  by
\begin{eqnarray*}
\left\{
\begin{split}
v_t&=g(w_t)F_t,
\\ F_t&=  \pi_{l}\tilde{B}(w_t,\zeta_t)-B_1\Delta \zeta_t^{l}+\nu\Delta \zeta_t^{l},
\end{split}\right.
\end{eqnarray*}
 here $ \pi_{l}\tilde{B}(w_t,\zeta_t)= P_N\tilde{B}(w_t,\zeta_t)$.  By Hypothesis \textbf{ H3}, $g$ depends on $N$. $\zeta_t$ also depends on $N.$

\subsection{The estimate of $\zeta_t$}
\begin{lemma}\label{dp-2}
For any    $\eta\leq \frac{\nu^2}{8B_0} $,
there exist some $N_0:=N_0(B_0,\eta,L_Q,\nu)>0$, such that if   $N\geq N_0$,
then
 \begin{eqnarray*}
    \mE\big[ \|\zeta_t\|^2 e^{\nu N^2t-4\eta \int_0^t \|w_r|_1^2\dif r}\big]\leq 1, \ \ \forall t>0.
 \end{eqnarray*}
Furthermore,   there exists a absolutely positive  constant  $C$ such that
 \begin{eqnarray*}
   \mE \|\zeta_t\|\leq C   e^{-\frac{1}{4}\nu N^2 t}\cdot e^{\frac{2\eta}{\nu}\|w_0\|^2}, ~~\forall t>0.
 \end{eqnarray*}
\end{lemma}
\begin{proof}
By  It\^o's formula,
\begin{eqnarray}
\nonumber
  \dif \|\zeta_t\|^2&\leq & 2 B_1 \langle \zeta_t,  \Delta \zeta_t^l  \rangle \dif t+
  2 \nu  \langle \zeta_t,  \Delta \zeta_t^h  \rangle \dif t
  +2 \langle\zeta_t^h,\tilde{B}(w_t,\zeta_t)\rangle \dif t
  \\  \label{dp-11}&&+h_t\dif  B_t+ L_Q\|\zeta_t\|^2\dif t.
\end{eqnarray}
Here we need to estimate $\langle\zeta_t^h,\tilde{B}(w_t,\zeta_t)\rangle $, observe that
\begin{eqnarray*}
 \langle\zeta_t^h,\tilde{B}(w_t,\zeta_t)\rangle  = \langle\zeta_t^h,B(\cK \zeta_t, w_t)\rangle + \langle\zeta_t^h,B(\cK w_t,\zeta_t)\rangle,
\end{eqnarray*}
and  (\ref{1-6}), for some $\hat{C},$
\begin{eqnarray}
 \nonumber  \langle\zeta_t^h,B(\cK \zeta_t, w_t)\rangle &\leq&  \hat{C}\|\zeta_t\|_{\frac{1}{2}}\|w_t\|_1\|\zeta_t^h\|
\\  \nonumber  &\leq & \eta \|w_t\|_1^2 \|\zeta_t^h\|^2+\frac{\hat{C}^2}{4 \eta}\|\zeta_t\|_{\frac{1}{2}}^2
\\   \nonumber &\leq &  \eta \|w_t\|_1^2 \|\zeta_t^h\|^2+\frac{\hat{C}^2}{4\eta}\|\zeta_t\|_1\|\zeta_t\|
 \\ \label{1-15} &\leq &\eta \|w_t\|_1^2 \|\zeta_t\|^2+
 \frac{\nu}{6} \|\zeta_t\|_1^2 +\frac{\hat{C}^4}{8 \eta^2 \nu}\|\zeta_t\|^2,
\end{eqnarray}
and,
\begin{eqnarray*}
  \langle\zeta_t^h,B(\cK w_t,\zeta_t)\rangle &=&   \langle\zeta_t^h,B(\cK w_t,\zeta_t^l )\rangle
  \\  &\leq &  \hat{C}\|w_t\|_{\frac{1}{2}}  \|\zeta_t^l\|_1\|\zeta_t^h\|
  \\ &\leq & \eta \|\zeta_t\|^2 \|w_t\|_1^2+\frac{\hat{C}^2}{4\eta} \|\zeta_t^l\|_1^2.
\end{eqnarray*}
Therefore, by  (\ref{dp-11}),
\begin{eqnarray*}
  \dif \|\zeta_t\|^2&\leq & -2B_1\|\zeta_t^l\|_1^2\dif t -2\nu \|\zeta_t^h\|_1^2\dif t
  +4\eta   \|\zeta_t\|^2 \|w_t\|_1^2\dif t
  \\ &&+ \frac{\hat{C}^2}{2\eta}   \|\zeta_t^l\|_1^2+\frac{\nu}{3} \|\zeta_t\|_1^2+\big[\frac{\hat{C}^4}{4 \eta^2 \nu}+L_Q\big]\|\zeta_t\|^2
  \\ &&+h_t\dif  B_t.
\end{eqnarray*}
Hence, if    $B_1\geq \nu N^2 \geq  \frac{\hat{C}^2}{2\eta} + \frac{\nu}{3}+\big[\frac{\hat{C}^4}{4 \eta^2 \nu}+L_Q\big] $ and
$\frac{2}{3}\nu N^2\geq   \big[\frac{\hat{C}^4}{4 \eta^2 \nu}+L_Q\big]$,
\begin{eqnarray}\label{1-22}
 \nonumber \dif \|\zeta_t\|^2&\leq & -B_1\|\zeta_t^l\|_1^2\dif t -\nu N^2 \|\zeta_t^h\|^2\dif t
  +4\eta   \|\zeta_t\|^2 \|w_t\|_1^2 +h_t\dif  B_t
  \\ \label{1-13} &\leq & -\nu N^2 \|\zeta_t\|^2\dif t  +4\eta   \|\zeta_t\|^2 \|w_t\|_1^2\dif t +h_t\dif  B_t,
\end{eqnarray}
from which,
\begin{eqnarray}\label{1-11}
  \mE\big[ \|\zeta_t\|^2 e^{\nu N^2t-4\eta \int_0^t \|w_r|_1^2\dif r}\big]\leq 1.
\end{eqnarray}
Therefore,  by H\"older inequality,   (\ref{1-11}) and Lemma  \ref{10} ,  if   $\eta\leq \frac{\nu^2}{8B_0} $  and  $\frac{\sqrt{8B_0\eta}}{\nu}\leq N$,
\begin{eqnarray*}
\mE \|\zeta_t\| &=& \mE\Big[ \|\zeta_t\| e^{-\frac{1}{2}\nu N^2t+2\eta \int_0^t \|w_r|_1^2\dif r}\cdot  e^{\frac{1}{2}\nu N^2t-2\eta \int_0^t \|w_r|_1^2\dif r}  \Big]
\\ &\leq & \left(\mE\big[ e^{-\nu N^2t+4\eta \int_0^t \|w_r|_1^2\dif r}\big]\right)^{\frac{1}{2}}
\\
&\leq & C  e^{-\frac{1}{4}\nu N^2 t}\cdot
 \left(\mE\big[ e^{-\frac{1}{2}\nu N^2t+4\eta \int_0^t \|w_r|_1^2\dif r}\big]\right)^{\frac{1}{2}}
 \\ &\leq &C   e^{-\frac{1}{4}\nu N^2 t}\cdot e^{\frac{2\eta}{\nu}\|w_0\|^2}.
\end{eqnarray*}
\end{proof}

\begin{lemma}\label{2-3}
  For any $n\in \mN$ and $\eta\leq \frac{\nu^2}{8nB_0}$,  there exists $N_0=N_0(n,\eta,B_0,\nu,L_Q)$, such that if $N\geq N_0$,
  \begin{eqnarray*}
  \mE\[\|\z_t\|^{2n}e^{\big(\nu nN^2-n(n-1)L_Q/2\big)t-\int_0^t4\eta\|w_r\|^2\dif r}\]\leq 1.
\end{eqnarray*}
and furthermore  for some absolute constant $C,$
\begin{eqnarray*}
 \mE\[\|\z_t\|^{n}\] &\leq &Ce^{\frac{2n\eta}{\nu}\|w_0\|^2}e^{-\nu nN^2 t/2}.
\end{eqnarray*}
\end{lemma}
\begin{proof}
By It\^o's formula,
\begin{eqnarray*}
  \dif \|\zeta_t\|^2&= & 2 B_1 \langle \zeta_t,  \Delta \zeta_t^l  \rangle \dif t+
  2 \nu  \langle \zeta_t,  \Delta \zeta_t^h  \rangle \dif t
  +2 \langle\zeta_t^h,\tilde{B}(w_t,\zeta_t)\rangle \dif t
  \\  &&+\langle \zeta_t, DQ(w_t)\zeta_t\dif B_t\rangle+ \|DQ(w_t)\zeta_t\|^2\dif t.
\end{eqnarray*}
and
\begin{eqnarray*}
  \dif  \|\z_t\|^{2n}=n\|\z_t\|^{2n-2} \dif \|\z_t\|^2+\frac{n(n-1)}{2}\|\z_t\|^{2n-4}h_t^2\dif t,
\end{eqnarray*}
so  by (\ref{1-22})
\begin{eqnarray*}
 \dif  \|\z_t\|^{2n}&\leq & n\|\z_t\|^{2n-2} \[-\nu N^2 \|\zeta_t\|^2\dif t  +4\eta   \|\zeta_t\|^2
 \|w_t\|_1^2\dif t +h_t\dif  B_t \]
 \\ &&+ \frac{n(n-1)L_Q}{2}\|\z_t\|^{2n}\dif t.
\end{eqnarray*}
Therefore,
\begin{eqnarray*}
  \mE\[\|\z_t\|^{2n}e^{\big(\nu nN^2-n(n-1)L_Q/2\big)t-\int_0^t4n\eta\|w_r\|^2\dif r}\]\leq 1.
\end{eqnarray*}
By H\"older inequality and the above inequality,
\begin{eqnarray*}
 \mE\[\|\z_t\|^{n}\]\leq \mE\[e^{\big(-\nu nN^2+n(n-1)L_Q/2\big)t+\int_0^t4n\eta\|w_r\|^2\dif r}\]^{\frac{1}{2}},
\end{eqnarray*}
here we use the notation, for any random variable $X$, $\mE\[X\]^{p}:=\[\mE X\]^{p}.$ Hence, by Lemma \ref{10},
there exists $N_0=N_0(n,\eta,B_0,\nu,L_Q)$, such that if $N\geq N_0$, then for some absolute constant $C,$
\begin{eqnarray*}
 \mE\[\|\z_t\|^{n}\] &\leq &\mE\[e^{\big(-\nu nN^2+n(n-1)L_Q/2\big)t+\int_0^t4n\eta\|w_r\|^2\dif r}\]^{\frac{1}{2}}
 \\
 &\leq &Ce^{\frac{2n\eta}{\nu}\|w_0\|^2}e^{-\nu nN^2t/2}.
\end{eqnarray*}

\end{proof}

\subsection{ The estimate of  $\mE\[ \big|\int_0^tv_s\dif B_s\big|^2\] $ }
\begin{eqnarray}
\left(\mathbb{E}\left|\int_0^t v(s)\dif W(s)\right|^2 \right)= \int_0^t\mathbb{E}\|v(s)\|^2\dif s\leq C\int_0^t\mathbb{E}\|F_s\|^2\dif s.
\end{eqnarray}
Remaking that, by the definition of $F_t$ and
 ~$\|\pi_l \tilde{B}(u,w)\| \leq C(N) \cdot
\|u\|\cdot \|w\|$    for some
constant $C=C(N)$(see  \cite[Lemma A.4]{mattingly}),
\begin{eqnarray*}
\mathbb{E}\|F_s\|^2 &\leq& C \cdot \(\mE\|\zeta_s\|^2+\mE \big[\|w_s\|^2\|\zeta_s\|^2\big]\)
\\ &\leq & C \cdot \(\mE\|\zeta_s\|^2+\[\mE \|w_s\|^4\]^{1/2} \[\mE \|\zeta_s\|^4\]^{1/2}\)
\end{eqnarray*}
Hence, by H\"older inequality,  Lemma \ref{16} and   Lemma    \ref{2-3}, for any  $\eta\leq  \frac{\nu^2}{32B_0}$, there exist some $N_0:=N_0(B_0,\eta,L_Q,\nu)>0$, such that if     $N\geq N_0$,
then there exists
\begin{eqnarray}
\label{bo-s-3-1}
\mathbb{E}\left| \int_0^{\infty}v(s)\dif B(s)\right|^2 \leq C(N) \exp{\(\big(\frac{4\eta}{\nu} +\eta \big)\|w_0\|^2\)}.
\end{eqnarray}

\subsection{The proof of Theorem \ref{1-17}}
\begin{proof}
For any $\xi$ with $\|\xi\|=1,$ for some constant $C=C(\nu,B_0,N,\eta)$,
\begin{eqnarray*}
&& \langle  \nabla P_t \varphi (w_0),\xi \rangle
\\ &&=\mathbb{E}_{w_0}(\langle  \nabla  \varphi (w_t),\xi \rangle)
= \mathbb{E}_{w_0}(  (\nabla  \varphi) (w_t)J_{0,t}\xi )
\\&& =\mathbb{E}_{w_0}(  (\nabla  \varphi) (w_t) \mathcal{D}^vw_t   )+\mathbb{E}_{w_0}(  (\nabla  \varphi) (w_t)\rho_t )
\\&& =\mathbb{E}_{w_0}( \mathcal{D}^{v} \varphi(w_t) )+\mathbb{E}_{w_0}(  (\nabla  \varphi) (w_t)\rho_t )
\\ && = \mE_{w_0}\[  \varphi(w_t) \int_0^tv_s\dif B_s\]+\mathbb{E}_{w_0}(  (\nabla  \varphi) (w_t)\rho_t )
\\ &&\leq   \[\mE_{w_0}   \varphi(w_t)^2\]^{\frac{1}{2}} \[\mE_{w_0} (\int_0^tv_s \dif B_s)^2\]^{\frac{1}{2}}
+ \[\mE_{w_0}\big[ (\nabla  \varphi) (w_t)  \big]^2\]^{\frac{1}{2}}
\[\mE_{w_0}\|\rho_t\|^2\]^{\frac{1}{2}}
\\ &&\leq  C(N) \exp{\(\big(\frac{4\eta}{\nu} +\eta \big)\|w_0\|^2\)}\sqrt{P_t|\varphi|^2(w_0)}+
Ce^{\frac{4\eta}{\nu}\|w_0\|^2}e^{-\nu N^2 t} \sqrt{P_t\|D\varphi\|^2(x)}
\end{eqnarray*}
\end{proof}

\section{Proof of Exponential Mixing}
For getting the exponential convergence, we using the methods in
\cite{Hairer02}.
In the Assumption 4.1, 4.2, 4.3 and Theorem $\ref{5-9}$  below, we assume that we are given a random flow $\Phi_t$ on a Banach space $H$. We will assume that the map $x \mapsto \Phi_t(\omega,x)$ is $\mathcal{C}^{1}$ for almost every element $\omega$ of the underlying probability space. We will denote by $D\Phi_t$ the Fr\'{e}chet derivative of $\Phi_t(\omega,x)$ with respect to x.

Let  $\mathcal{C}(\mu_1,\mu_2)$ for the set of all measures $\Gamma$
on $H \times H$ such that $\Gamma(A \times H)=\mu_1(A)$ and
$\Gamma(H\times A)=\mu_2(A)$ for every Borel set~$ A \subset H.$ The
following three assumptions are  from \cite{Hairer02}.
\begin{assumption}\label{5-1} There exists a function $V:H\rightarrow [1,\infty)$~with
the following properties:

 \begin{enumerate}
\item  There exists two strictly increasing continuous functions $V^{*}$~and $V_{*}$~
from $[0,\infty)\rightarrow [1,\infty)$ such  that
\begin{eqnarray}\label{5-2}
V_{*}(\|x\|)\leq V(x) \leq V^{*}(\|x\|)
\end{eqnarray}
for all $x \in H$ and such that $\lim_{a\rightarrow \infty}V_{*}(a)=\infty.$
\item There exists constants C and $\kappa\geq 1$ such that
\begin{eqnarray}\label{5-3}
aV^{*}(a)\leq CV_{*}^{\kappa}(a),
\end{eqnarray}
for every $a>0$.
\item There exists a positive constants C, $r_0 <1$, a decreasing function $\zeta:[0,1]\rightarrow [0,1]$~
with $\zeta(1)<1$~such that for every $h \in H $  with $\|h\|=1$
    \begin{eqnarray}\label{5-4}
    \mathbb{E}V^{r}(\Phi_t(x))(1+\|D\Phi_t(x)h)\|)\leq CV^{r\zeta(t)}(x),
    \end{eqnarray}
    for every $x \in  H $, every  $r \in [r_0,\kappa]$, and every $t \in [0,1].$
\end{enumerate}
\end{assumption}

If Assumption $\ref{5-1}$ is satisfied, then for every Fr\'{e}chet differentiable function $\varphi: H\rightarrow R$, we introduce the following norm
\begin{eqnarray*}
\|\varphi \|_{V}=\sup_{x\in H}\frac{|\varphi(x)|+\|D\varphi(x)\|}{V(x)},
\end{eqnarray*}
and  for $r \in(0,1]$, a family of
distance $\rho_r$ on $H$ is defined by
$$
\rho_r=\inf_{\gamma} \int_0^1V^{r}(\gamma(t))\|\dot{\gamma}(t)\|dt,
$$
where the infimum runs over all paths $\gamma$~such that
$\gamma(0)=x$~and $\gamma(1)=y$. For simple, we will write $\rho$
for $\rho_1$.

\begin{assumption}\label{5-5}
There exists a $C_1>0$ and $p\in[0,1)$ so that for every $\alpha \in (0,1)$ there exists positive $T(\alpha)$
 and $C(\alpha)$ with
\begin{eqnarray}\label{5-6}
\|DP_t \varphi(x)\|\leq C_1 V^{p}(x)\left(C(\alpha)\sqrt{(P_t |\varphi|^2)(x)} +\alpha\sqrt{(P_t\|D\varphi\|^2)(x)} \right),
\end{eqnarray}
for every $x \in H$ and $t \geq T(\alpha)$.
\end{assumption}

\begin{assumption}\label{5-7}
Given any $C>0,\ r\in (0,1)$ and $\delta>0$,  there exists a $T_0$ so that for any $T \geq T_0$
there exists an  $a>0$ so that
\begin{eqnarray}\label{5-8}
\inf_{|x|,|y|\leq C} \sup_{\Gamma \in \mathcal{C}(P_T^{*}\delta_x,P_T^{*}\delta_y)}\Gamma\left\{(x',y')
\in H \times H:\rho_r(x',y')<\delta\right\}\geq a.
\end{eqnarray}
\end{assumption}

 If  the setting of the semigroup $P_t$ possesses an invariant measure $\mu_{*}$,  we  define
\begin{eqnarray*}
\|\varphi \|_{\rho}=\sup_{x \neq y}\frac{|\varphi(x)-\varphi(y) |}{\rho(x,y)}+\left|\int_{H}\varphi(x) \mu_{*}(dx)\right|.
\end{eqnarray*}

The next Theorem comes from Theorem $3.6$, Corollary 3.5  and
Theorem $4.5$  in \cite{Hairer02}.
\begin{thm}\label{5-9}
Let $\Phi_t$ be a stochastic flow on a Banach space $H$ which is almost surely $\mathcal{C}^{1}$
and satisfy Assumption $\ref{5-1}$. Denote by $P_t$ the corresponding Markov semigroup and assume that
it satisfies Assumption $\ref{5-5}$ and $\ref{5-7}$.   Then there exists a unique invariant probability
measure $\mu_*$ for $P_t$ and   exists constants $\gamma>0$ and $C>0$ such that
\begin{eqnarray*}
\|P_t\varphi- \mu_{*}\varphi\|_{\rho} &\leq &Ce^{-\gamma t}\|\varphi-\mu_{*}\varphi\|_{\rho},
\\
\|P_t\varphi- \mu_{*}\varphi\|_{V} &\leq& Ce^{-\gamma t}\|\varphi-\mu_{*}\varphi\|_{V},
\end{eqnarray*}
for every Fr$\acute{e}$chet differentiable function $\varphi: H\rightarrow \mathbb{R}$ and every $t>0.$
\end{thm}

The next Lemma comes from Lemma $5.1$  in \cite{Hairer02}.
\begin{lemma}\label{5-10}
Let U be a real-valued semi-martingale
$$
dU(t,\omega)=F(t,\omega)dt+G(t,\omega)dB(t,\omega),
$$
where B is a standard Brownian motion.  Assume that there exists a
process Z and positive constants $b_1,b_2,b_3$, with $b_2>b_3,$   such
that $F(t,\omega) \leq b_1-b_2Z(t,\omega),U(t,\omega) \leq
Z(t,\omega),$   ~and $G(t,\omega)^2\leq b_3Z(t,\omega)$ almost
surely.   Then the bound
$$
\mathbb{E}\exp\left(U(t)+\frac{b_2e^{-b_2t/4}}{4}\int_0^tZ(s)ds\right)\leq \frac{b_2\exp(\frac{2b_1}{b_2})}{b_2-b_3}\exp\left(U(0)e^{-\frac{b_2}{2}t}\right),
$$
holds for every $t \geq 0.$
\end{lemma}

\begin{proposition}  Let  $H = L^2_0$ be  the space of real-valued square-integrable functions on the torus $[-\pi,\pi]^2$  with vanishing mean, and the  random flow $\Phi_t$ on the  space $H$  is  given by
the solution to  $(\ref{1-2})$.
 Under   the conditions of Theorem
$\ref{5-9}$,       Assumption  $\ref{5-1}$ holds with
$$ V(x)= V_{\eta_0}(x)=e^{\eta_0 \|x\|^2},\  \ \ \eta_0=\frac{\nu}{16B_0}.$$
\end{proposition}
\begin{proof}
By It\^{o}
formula,
\begin{equation*}
d  \eta \Arrowvert  w_{t} \Arrowvert^2+2\eta \nu \| w_{t} \Arrowvert^2_1dt
=2\eta \langle w_t,Q(w_t)dB_t\rangle +\eta \| Q(w_t) \|^2dt.
\end{equation*}
From Lemma $\ref{5-10}$,
 $$
\mathbb{E}\exp\left(U(t)+\frac{b_2e^{-b_2t/4}}{4}\int_0^tZ(s)ds\right)\leq \frac{b_2\exp(\frac{2b_1}{b_2})}{b_2-b_3}
\exp\left(U(0)e^{-\frac{b_2}{2}t}\right).
$$
where $U(t)=\eta \|w_t\|^2,\ Z(t)=\eta \|w_t\|_1^2,\ b_1=\eta B_0,\ b_2=2\nu, \ b_3=4\eta B_0$.
Therefore  when  $\eta \leq \frac{\nu}{4B_0}$,
\begin{eqnarray}\label{5-12}
\mathbb{E}\exp\left(\eta \|w_t\|^2+\frac{2\nu e^{-2\nu t/4}}{4}\int_0^t\eta \|w_s\|_1^2 ds\right)
\leq 2\exp(\frac{\eta B_0}{\nu})\exp\left(\eta \|w_0\|^2e^{-\frac{2\nu}{2}t}\right).
\end{eqnarray}
For $\|\xi\|=1$,denote $\xi_t=J_t\xi=Dw_t^x \xi$, where $x$ is the
initial value and $D$ is the differential operator with $x$. So $\xi_t$
satisfies  the following equation
 \begin{eqnarray}\label{5-11}
d\xi_t=\nu \Delta \xi_t dt +\tilde{B}(w_t,\xi_t)dt+DQ(w_t)\xi_t dB_t,
\end{eqnarray}
and thus
\begin{eqnarray*}
\dif \|\xi_t\|^2 \leq  -2 \nu \|\xi_t\|_1^2\dif t+2\langle B(\mathcal{K}\xi_t,w_t),\xi_t\rangle \dif t+L_Q\|\xi_t\|^2 dt+h_t\dif W_t.
\end{eqnarray*}
By the  the similar  step to get (\ref{1-15}), one arrives at
\begin{eqnarray*}
2\langle B(\mathcal{K}\xi_t,w_t),\xi_t\rangle \leq \eta \|w_t\|_1^2\|\xi_t\|^2+\nu \|\xi_t\|_1^2+\frac{16\hat{C}^4}{\eta^2 \nu}\|\xi_t\|^2,
\end{eqnarray*}
and
\begin{eqnarray*}
d\|\xi_t\|^2 \leq \eta \|w_t\|_1^2\|\xi_t\|^2dt+(\frac{16\hat{C}^4}{\eta^2 \nu}+L_Q)\|\xi_t\|^2dt+h_tdW_t.
\end{eqnarray*}
Define the function $h(\eta)=(\frac{16\hat{C}^4}{\eta^2 \nu}+L_Q),$
from the above inequality one arrives at
\begin{eqnarray}\label{5-13}
\mathbb{E}\Big\{\|\xi_t\|^2 \exp{(-h(\eta)t-\int_0^t \eta \|w_s\|_1^2ds)}\Big\}\leq 1,\forall \eta>0.
\end{eqnarray}
Set $b=e^{-\frac{\nu}{2}},$  $\eta \leq \frac{\nu}{8B_0},\ t\in [0,1]$.
From $(\ref{5-12})$ and $(\ref{5-13}),$
\begin{eqnarray*}
\mathbb{E}\big\{\exp{(\eta \|w_t\|^2)}\|\xi_t\|\big\}&=&\mathbb{E}\[\exp{(\eta \|w_t\|^2)}
\exp{(\frac{b \eta \nu }{2}\int_0^t\|w_s\|_1^2ds)} \cdot \|\xi_t\|\exp{(-\frac{b \eta \nu}{2}\int_0^t\|w_s\|_1^2ds)}\]
\\&\leq &  \Big(\mathbb{E}\exp{(2\eta  \|w_t\|^2+b \eta \nu \int_0^t\|w_s\|_1^2ds)}\Big)^{\frac{1}{2}}
 \Big( \mathbb{E}\big[ \|\xi_t\|^2\exp{(-b \eta \nu \int_0^t\|w_s\|_1^2ds)}\big]\Big)^{\frac{1}{2}}
\\&\leq   & \Big(2\exp(\frac{2\eta B_0}{\nu})
\exp\big(2\eta \|w_0\|^2e^{-\frac{2\nu}{2}t}\big)\Big)^{\frac{1}{2}} \exp{(\frac{h(b\eta \nu)}{2}t)} \\
&=&   C(\eta, B_0,\nu)\exp\Big(\eta \|w_0\|^2e^{-\nu t}\Big)
  \exp{(\frac{h(b\eta \nu)}{2}t)}.
  \end{eqnarray*}
Set $\eta_0=\frac{\nu}{16 B_0}$, we  know that the above inequality
is satisfied for all $\eta \in[0,2\eta_0]$.  So $w_t$ satisfies
Assumption $\ref{5-1}$ for $V(x)=e^{\eta_0 \|x\|^2},\ \kappa=2,\
r_0=\frac{1}{2}$ , $V_*(a)=V^*(a) =e^{\eta_0 a^2}$ and $\zeta_t=e^{-\frac{1}{2}\nu t}$.
\end{proof}
\appendix
\section{ }
Let $r=\tilde{w}(t,B,w_0^2)-w(t,B,w_0^1).$
\begin{proposition}\label{p-3}
Assume \textbf{H1} and \textbf{H2} hold. There exists $\gamma_1,\gamma_2>0$, $K_0>0$ and  $N_0=N_0(B_0,\nu,L_Q)$   such that for any $K\geq K_0$ and $N\geq N_0$ and any $(t,w_0^1, w_0^2)\in (0,\infty)\times H\times H$,
\begin{eqnarray*}
 \mE\[\|r\|^2\]\leq 2  e^{\gamma_1 \|w_0^1\|^2}\|r(0)\|^{2}e^{-\gamma_2 t}.
\end{eqnarray*}
here $r(t)=\tilde{w}(t,B,w_0^2)-w(t,B,w_0^1)=\tilde{w}-w.$
\end{proposition}
\begin{proof}
For any function $f,$ denote $\delta f=f(\tilde{w})-f(w).$
Taking the fifference between (\ref{p-1}) and (\ref{p-2}), we obtain
\begin{eqnarray*}
  \dif r &=& \nu \Delta r\dif t-KP_Nr\dif t+\delta Q(w)\dif B_t+\[B(\cK \tilde{w},\tilde{w})-B(\cK w,w)\]\dif t
  \\ &=& \nu \Delta r\dif t-KP_Nr\dif t+\delta Q(w)\dif B_t+\[B(\cK \tilde{w}, r)+B(\cK r,w)\] \dif t.
\end{eqnarray*}
Apply It\^o's formula to $\|r\|^2$, we have
\begin{eqnarray*}
  \dif \|r\|^2=-2\nu \|r\|_1^2\dif t -2K\|P_Nr\|^2\dif t+2\langle B(\cK r,w), r\rangle\dif t +2\langle r, \delta Q(w) \dif B_t\rangle+\|\delta Q(w)\|^2_{\cL_2(U,H)}.
\end{eqnarray*}
Remarking that for any $\eta>0,$
\begin{eqnarray*}
 \langle B(\cK r,w), r\rangle & \leq &  \|w\|_1 \cdot  \|r\|\cdot \|r\|_{1/2}
 \\ &\leq & \eta \|w\|_1^2\|r\|^2+C(\eta) \|r\|_{1/2}^2
 \\ &\leq & \eta \|w\|_1^2\|r\|^2+\frac{\nu}{6}\|r\|_1^2+C(\eta)\|r\|^2
\end{eqnarray*}
then
\begin{eqnarray*}
  \dif \|r\|^2 &\leq &  -\frac{3\nu}{2} |r|_1^2\dif t -2K|P_Nr|^2\dif t+  C\|w\|_1^2 \|r\|^2  \dif t +2\langle r, \delta Q(w) \dif B_t\rangle+L_Q \|r\|^2
  \\ &\leq & \big[-\frac{3}{2}\nu N^2 +\eta \|w_t\|_1^2 \big]\|r\|^2\dif t + +2\langle r, \delta Q(w) \dif B_t\rangle,
\end{eqnarray*}
Integrating this formula and taking the expectation, if follows
\begin{eqnarray*}
  \mE\[ \|r\|^2\exp\{\nu N^2 t-\eta \int_0^t \|w_s\|_1^2\dif s\} \]\leq \|r(0)\|^2.
\end{eqnarray*}
Apply It\^o's formula to $\big[\|r\|^2\big]^2$, with a similar step, one arrives at
\begin{eqnarray*}
  \dif \|r\|^4\leq \big[-\nu N^2 +\eta \|w_t\|_1^2 \big]\|r\|^4\dif t + +2\langle r, \delta Q(w) \dif B_t\rangle,
\end{eqnarray*}
and
\begin{eqnarray*}
  \mE\[ \|r\|^4\exp\{\nu N^2 t-\eta \int_0^t \|w_s\|_1^2\dif s\} \]\leq \|r(0)\|^4.
\end{eqnarray*}
It follows from the above inequality and   Lemma \ref{10} that for some $\gamma_1,\gamma_2>0,$
\begin{eqnarray*}
 \mE\[\|r\|^2\]\leq 2C  e^{\gamma_1 \|w_0^1\|^2}\|r(0)\|^{2}e^{-\gamma_2 t}.
\end{eqnarray*}

\end{proof}










\bibliographystyle{elsarticle-num}



\end{document}